\documentclass[11pt]{article}

\usepackage{latexsym,color,amsmath,amsthm,amssymb,amscd,amsfonts}
\usepackage{graphicx}
\usepackage{cancel}

\newtheorem{conj}{Conjecture}[section]
\newtheorem{theorem}[conj]{Theorem}

\newtheorem{remark}[conj]{Remark}
\newtheorem{lemma}[conj]{Lemma}

\newcommand\independent{\protect\mathpalette{\protect\independent}{\perp}} 
\def\independent#1#2{\mathrel{\rlap{$#1#2$}\mkern2mu{#1#2}}}

\newcommand{\supp}{\mathrm{Supp}}

\newcommand{\R}{\mathbb{R}}

\renewcommand{\P}{\mathbb{P}}

\newcommand{\E}{\mathbb{E}}

\DeclareMathOperator{\Var}{Var}

\newcommand{\Z}{\mathcal{Z}}

\newcommand{\eps}{\varepsilon}

\renewcommand{\Z}{\mathbb{Z}}

\newcommand{\N}{\mathbb{N}}

\date{\vspace{-5ex}}

\author{Arnaud Marsiglietti and Puja Pandey}

\title{A Note on Statistical Distances for Discrete Log-Concave Measures}

\begin{document}

\maketitle

\begin{abstract}

    In this note we explore how standard statistical distances are equivalent for discrete log-concave distributions. Distances include total variation distance, Wasserstein distance, and $f$-divergences.

\end{abstract}


\section{Introduction}

The study of convergence of probability measures is central in probability and statistics, and may be performed via statistical distances for which the choice has its importance (see, e.g., \cite{GS02}, \cite{Rac}). The space of probability measures, say over the real numbers, is infinite dimensional, therefore there is a priori no canonical distance, and distances may not be equivalent. Nonetheless, an essential contribution made by Meckes and Meckes in \cite{MM} demonstrates that certain statistical distances between continuous log-concave distributions turn out to be equivalent up to constants that may depend on the dimension of the ambient space (see also \cite{CG} for improved bounds, and \cite{MP} for the extension to the broader class of so-called $s$-concave distributions).

The goal of this note is to develop quantitative comparisons between distances for discrete log-concave distributions. Let us denote by $\N = \{0, 1, 2, \dots\}$ the set of natural numbers and by $\Z$ the set of integers. Recall that the probability mass function (p.m.f.) associated with an integer valued random variable $X$ is
$$ p(k) = \mathbb{P}(X = k), \quad k \in \Z. $$
An integer-valued random variable $X$ is said to be log-concave if its probability mass function $p$ satisfies
$$ p(k)^{2} \geq p(k-1)p(k+1) $$
for all $k \in \mathbb{Z}$ and the support of $X$ is an integer interval. 

Discrete log-concave distributions form an important class. Examples include Bernoulli, discrete uniform, binomial, geometric and Poisson distributions. We refer to \cite{Sta}, \cite{Bre94}, \cite{SW}, \cite{Bra15}  for further background on log-concavity.

Let us introduce the main distances we will work with (we refer to \cite{LV}, \cite{Dud}, \cite{GS02}, \cite{Rac} for further background on statistical distances). Our setting is the real line $\R$ equipped with its usual Euclidean structure $d(x,y) = |x-y|$, $x,y \in \R$.

\begin{enumerate}

\item The bounded Lipschitz distance between two probability measures $\mu$ and $\nu$ is defined as
$$ d_{BL}(\mu,\nu) = \sup_{\|g\|_{BL} \leq 1} \left| \int g \, d\mu - \int g \, d\nu  \right|, $$
where for a function $g \colon \R \to \R$,
$$ \|g\|_{BL} = \max \left\{ \|g\|_{\infty}, \, \sup_{x \neq y} \frac{|g(x)-g(y)|}{|x-y|} \right\}. $$

  \item The L\'evy-Prokhorov distance between two probability measures $\mu$ and $\nu$ is defined as
\begin{equation*}
    d_{LP}(\mu,\nu) = \inf \left\{ \epsilon > 0 : \mu(A) \leq \nu(A^{\eps}) + \epsilon \mbox{ for all Borel set } A \subset \R \right\},
\end{equation*}
where $A^{\eps} = \{x \in \R : d(x,A) < \eps\}$.

\hspace*{5mm} Using the Ky-Fan distance, which is defined for two random variables $X$ and $Y$ as
$$ K(X,Y) = \inf \{ \eps > 0 : \P(|X-Y| > \eps) < \eps \}, $$
the L\'evy-Prokhorov distance admits the following coupling representation,
\begin{equation}\label{LV-coupl}
d_{LP}(\mu,\nu)= \inf K(X,Y),
\end{equation}
where the infimum runs over all random variables $X$ with distribution $\mu$ and random variables $Y$ with distribution $\nu$ (see, e.g., \cite{Rac}).

    \item The total variation distance between two probability measures $\mu$ and $\nu$ is defined as
\begin{align*}
    d_{TV}(\mu,\nu) = 2 \sup_{A \subset \R} |\mu(A) - \nu(A)|.
\end{align*}

\hspace*{5mm} The total variation distance admits the following coupling representation,
\begin{equation}\label{d-tv-coupl}
d_{TV}(\mu,\nu) = \inf \P(X \neq Y),
\end{equation}
where the infimum runs over all random variables $X$ with distribution $\mu$ and random variables $Y$ with distribution $\nu$ (see, e.g., \cite{GS02}). Moreover, for integer valued measures, one has the following identity,
\begin{equation}\label{d-tv-dis}
d_{TV}(\mu,\nu) = \sum_{k \in \Z} |\mu(\{k\}) - \nu(\{k\})|.
\end{equation}

\item The $p$-th Wasserstein distance, $p \geq 1$, between two probability measures $\mu$ and $\nu$ is defined as
    \begin{align*}
        W_{p}(\mu,\nu) = \inf \mathbb{E}[|X-Y|^{p}]^{\frac{1}{p}},
    \end{align*}
    where the infimum runs over all random variables $X$ with distribution $\mu$ and random variables $Y$ with distribution $\nu$.

\item Let $f \colon [0, +\infty) \to \R$ be a convex function such that $f(1) = 0$. The $f$-divergence between two probability measures $\mu$ and $\nu$ on $\Z$ is defined as
$$ d_f(\mu || \nu) = \sum_{k \in \Z} \nu(\{k\}) f\left( \frac{\mu(\{k\})}{\nu(\{k\})} \right). $$

Note that the choice of convex function $f(x) = x\log(x)$, $x \geq 0$, leads to the Kullback-Leibler divergence
$$ D(\mu || \nu) = \sum_{k \in \Z} \mu(\{k\}) \log \left( \frac{\mu(\{k\})}{\nu(\{k\})}  \right), $$
the function $f(x) = (x-1)^2$ yields the so-called $\chi^2$-divergence, while $f(x) = |x-1|$ returns us to the total variation distance.

\end{enumerate}

Let us review the known relationships between the above distances. It is known \cite[Corollaries 2 and 3]{Dud68} that bounded Lipschitz and L\'evy-Prokhorov distances are equivalent,
$$ \frac{1}{2} d_{BL}(\mu,\nu) \leq d_{LP}(\mu,\nu) \leq \sqrt{\frac{3}{2} d_{BL}(\mu,\nu)}. $$
One also has
$$ d_{LP}(\mu,\nu) \leq d_{TV}(\mu,\nu), $$
and, for $\mu, \nu$ integer valued probability measures,
$$ d_{TV}(\mu,\nu) \leq W_1(\mu,\nu), $$
see \cite{GS02}. By H\"older's inequality, if $p \leq q$, then
$$ W_{p}(\mu,\nu) \leq W_{q}(\mu,\nu). $$
As for divergences, the Pinsker-Csisz\'ar inequality (\cite{Pin}, \cite{C}) states that
$$ d_{TV}(\mu,\nu) \leq \sqrt{2 D(\mu || \nu)}. $$
Also, one has
$$ D(\mu || \nu) \leq \log(1 + \chi^2(\mu || \nu)). $$

The article is organized as follows. In Section \ref{2}, we establish properties for log-concave distributions on $\Z$ that are of independent interests. In section \ref{3}, we present our results and proofs.

\section{Preliminaries}\label{2}

In this section we gather the main tools used throughout the proofs. First, recall that a real-valued random variable $X$ with finite first moment is said to be centered if
$$ \E[X] = 0. $$
We start with an elementary lemma that will allow us to pass results for log-concave distributions on $\N$ to log-concave distributions on $\Z$, however with sub-optimal constants.

\begin{lemma}\label{sym}

If $X$ is symmetric log-concave on $\Z$, then $|X|$ is log-concave on $\N$.

\end{lemma}

\begin{proof}
Denote by $p$ (resp. $q$) the p.m.f. of $X$ (resp. $|X|$). Then, $q(0) = p(0)$ and $q(k) = 2p(k)$ for $k \geq 1$. Therefore, 
$$ q^2(1) = 4p^2(1) \geq 4 p(0) p(2) = 2q(0) q(2), $$
and for all $k \geq 2$,
$$ q^2(k) = 4p^2(k) \geq 4 p(k+1) p(k-1) = q(k+1) q(k-1). $$
Hence, $q$ is log-concave.
\end{proof}

We note that Lemma \ref{sym} no longer holds for non-symmetric log-concave random variables, as can be seen by taking $X$ supported on $\{-1,0,1,2,3\}$ with distribution $\P(X=-1) = \P(X=3) = 0.1$, $\P(X=0) = \P(X=2) = 0.2$, and $\P(X=1) = 0.4$. In this case, $\P(|X| = 2)^2 < \P(|X|=1)\P(|X|=3)$.

The next lemma provides moments bounds for log-concave distributions on $\Z$.

\begin{lemma}\label{moment-lem}

If $X$ is log-concave on $\Z$, then for all $\beta \geq 1$,
$$ \E[|X - \E[X]|^{\beta}]^{\frac{1}{\beta}} \leq \Gamma(\beta + 1)^{\frac{1}{\beta}} (2\E[|X - \E[X]|] + 1). $$

\end{lemma}

\begin{proof}
    It has been shown in \cite[Corollary 4.5]{MM22} that for all log-concave random variable $X$ on $\N$, for all $\beta \geq 1$,
    \begin{equation}\label{moment}
    \E[X^{\beta}]^{\frac{1}{\beta}} \leq \Gamma(\beta + 1)^{\frac{1}{\beta}} (\E[X] + 1).
    \end{equation}
    Let $X$ be a symmetric log-concave random variable on $\Z$, then by Lemma \ref{sym}, $|X|$ is log-concave on $\N$ so one may apply inequality \eqref{moment} to obtain
    \begin{equation}\label{moment-sym}
    \E[|X|^{\beta}]^{\frac{1}{\beta}} \leq \Gamma(\beta + 1)^{\frac{1}{\beta}} (\E[|X|] + 1).
    \end{equation}
    Now, let $X$ be a log-concave random variable on $\Z$. Let $Y$ be an independent copy of $X$, so that $X-Y$ is symmetric log-concave. Applying inequality \eqref{moment-sym}, we deduce that
    $$ \E[|X - \E[X]|^{\beta}]^{\frac{1}{\beta}} \leq \E[|X-Y|^{\beta}]^{\frac{1}{\beta}} \leq \Gamma(\beta + 1)^{\frac{1}{\beta}} (\E[|X-Y|] + 1) \leq \Gamma(\beta + 1)^{\frac{1}{\beta}} (2\E[|X - \E[X]|] + 1), $$
    where the first inequality follows from H\"older's inequality and the last inequality from triangle inequality.
\end{proof}

Let us derive concentration inequalities for log-concave distributions on $\Z$.

\begin{lemma}\label{concentration-lem}

For each log-concave random variable $X$ on $\Z$, one has for all $t \geq 0$,
$$ \P(|X - \E[X]| \geq t) \leq 2e^{- \frac{t}{2(2\E[|X - \E[X]|] + 1)}}. $$

\end{lemma}

\begin{proof}
    The proof is a standard application of the moments bounds obtained in Lemma \ref{moment-lem} (see, e.g., \cite{V22}). For $\lambda > 0$,
    \begin{eqnarray*}
    \E[e^{\lambda |X - \E[X]|}] = 1 + \sum_{\beta \geq 1} \frac{\lambda^{\beta}}{\beta !} \E[|X - \E[X]|^{\beta}] & \leq & 1 + \sum_{\beta \geq 1} \frac{\lambda^{\beta}}{\beta !} \beta ! (2\E[|X - \E[X]|] + 1)^{\beta} \\ & = & \sum_{\beta \geq 0} \left[ \lambda (2\E[|X - \E[X]|] + 1) \right]^{\beta} \\ & = & \frac{1}{1-\lambda (2\E[|X - \E[X]|] + 1)},
    \end{eqnarray*}    
    where the last identity holds for all $0 < \lambda < \frac{1}{2\E[|X - \E[X]|] + 1}$. Choosing $\lambda = \frac{1}{2(2\E[|X - \E[X]|] + 1)}$ yields
    $$ \E[e^{\lambda |X - \E[X]|}] \leq 2. $$
    Therefore, by Markov's inequality,
    $$ \P(|X - \E[X]| \geq t) = \P(e^{\lambda |X - \E[X]|} \geq e^{\lambda t}) \leq \E[e^{\lambda |X - \E[X]|}] e^{-\lambda t} \leq 2e^{- \frac{t}{2(2\E[|X - \E[X]|] + 1)}}. $$
\end{proof}

The following lemma, which provides a bound on the variance and maximum of the probability mass function of log-concave distributions on $\Z$, was established in \cite{BMM} and \cite{Ara} (see, also, \cite{BC15}, \cite{JMNS}).

\begin{lemma}[\cite{BMM}, \cite{Ara}]\label{maximum-bound}

Let $X$ be a log-concave distribution on $\Z$ with probability mass function $p$, then
$$ \sqrt{1 + \Var(X)} \leq \frac{1}{\|p\|_{\infty}} \leq \sqrt{1+ 12 \Var(X)}. $$

\end{lemma}

The following lemma is standard in information theory and provides an upper bound on the entropy of an integer valued random variable (see \cite{Ma88}). Recall that the Shannon entropy of an integer valued random variable $X$ with p.m.f. $p$ is defined as
$$ H(X) = \E[-\log(p(X))] = - \sum_{k \in \Z} p(k) \log(p(k)). $$

\begin{lemma}[\cite{Ma88}]\label{bound-ent}

For any integer valued random variable $X$ with finite second moment,
$$ H(X) \leq \frac{1}{2} \log \left( 2 \pi e \left( \Var(X) + \frac{1}{12} \right) \right). $$

\end{lemma}

The last lemma of this section provides a bound on the second moment of the information content of a log-concave distribution on $\Z$.

\begin{lemma}\label{second-bound}

Let $X$ be a discrete log-concave random variable on $\Z$ with probability mass function $p$. Then,
$$ \E[\log^2(p(X))] \leq 4 \left( 4 e^{-2} + 1 + \frac{H^2(X)}{\|p\|_{\infty}} \right). $$

\end{lemma}

\begin{proof}
Let $X$ be a log-concave random variable with p.m.f. $p$. Then $p$ is unimodal, that is, there exists $m \in \Z$ such that for all $k \leq m$, $p(k-1) \leq p(k)$ and for all $k \geq m$, $p(k) \geq p(k+1)$. Note that $p(m) = \|p\|_{\infty}$. Define, for $k \in \Z$,
$$ p^{\nearrow}(k) = \frac{p(k)}{\sum_{l \leq m} p(l)} 1_{\{k \leq m\}}, $$
and
$$ p^{\searrow}(k) = \frac{p(k)}{\sum_{l \geq m} p(l)} 1_{\{k \geq m\}}. $$
Note that both $p^{\nearrow}$ and $p^{\searrow}$ are monotone log-concave probability mass functions. Denote by $X^{\nearrow}$ (resp. $X^{\searrow}$) a random variable with p.m.f. $p^{\nearrow}$ (resp. $p^{\searrow}$). Denote also $a = \sum_{l \leq m} p(l)$ and $b=\sum_{l \geq m} p(l)$. On one hand, by a result of Melbourne and Palafox-Castillo \cite[Theorem 2.5]{MP-C},
$$ \Var(\log(p^{\nearrow}(X^{\nearrow}))) \leq 1, \qquad \Var(\log(p^{\searrow}(X^{\searrow}))) \leq 1. $$
On the other hand,
$$ H(X^{\nearrow}) = \sum_{k \leq m} \frac{p(k)}{a} \log \left( \frac{a}{p(k)} \right) \leq \frac{1}{a} \sum_{k \leq m} p(k) \log \left( \frac{1}{p(k)} \right) \leq \frac{1}{a} H(X), $$
and similarly,
$$ H(X^{\searrow}) \leq \frac{H(X)}{b}. $$
Therefore,
$$ \E[\log^2(p^{\nearrow}(X^{\nearrow}))] = \Var(\log(p^{\nearrow}(X^{\nearrow}))) + H^2(X^{\nearrow}) \leq 1 + \frac{H^2(X)}{a^2}, $$
and similarly, 
$$ \E[\log^2(p^{\searrow}(X^{\searrow}))] \leq 1 + \frac{H^2(X)}{b^2}. $$
We deduce that
\begin{eqnarray*}
\E[\log^2(p(X))] & = & \sum_{k \in \Z} p(k) \log^2(p(k)) \\ & \leq & \sum_{k \in \Z} ap^{\nearrow}(k) \log^2(ap^{\nearrow}(k)) +  \sum_{k \in \Z} bp^{\searrow}(k) \log^2(bp^{\searrow}(k)) \\ & \leq & 2 \left( a \log^2(a) + a\E[\log^2(p^{\nearrow}(X^{\nearrow}))] + b \log^2(b) + b\E[\log^2(p^{\searrow}(X^{\searrow}))] \right) \\ & \leq & 2 \left(4 e^{-2} + a + \frac{H^2(X)}{a} + 4 e^{-2} + b + \frac{H^2(X)}{b} \right) \\ & \leq & 4 \left( 4 e^{-2} + 1 + \frac{H^2(X)}{\|p\|_{\infty}} \right),
\end{eqnarray*}
where we used the fact that $a,b \in [\|p\|_{\infty}, 1]$. 
\end{proof}

\begin{remark}\label{second-iso}

For a log-concave random variable $X$ on $\Z$ with probability mass function $p$ and variance bounded by 1, the above bounds imply
\begin{eqnarray}\label{moment-iso}
\E[|X - \E[X]|^{\beta}]^{\frac{1}{\beta}} & \leq & 3 \Gamma(\beta + 1)^{\frac{1}{\beta}}, \quad \beta \geq 1, \\ \label{concentration-iso} \P(|X - \E[X]| \geq t) & \leq & 2 e^{- \frac{t}{6}}, \quad t \geq 0, \\ \label{infinity-iso} \frac{1}{\|p\|_{\infty}} & \leq & \sqrt{13}, \\ \label{entropy-iso} H(X) & \leq & \frac{1}{2} \log \left( 2 \pi e \left( 1 + \frac{1}{12} \right) \right) \leq \frac{3}{2},
\end{eqnarray}
in particular, we also deduce
\begin{equation}\label{varent-iso}
\E[\log^2(p(X))] \leq 4(4e^{-2} + 1 + \frac{9}{4} \sqrt{13}) \leq 39.
\end{equation}

\end{remark}

\section{Main results and proofs}\label{3}

This section contains our main results together with the proofs. The first theorem establishes quantitative reversal bounds between 1-Wasserstein distance and L\'evy-Prokhorov distance.

\begin{theorem}\label{w1}
    Let $\mu$ and $\nu$ be centered log-concave probability measures on $\Z$, then
     $$ W_1(\mu, \nu) \leq 4 d_{LP}(\mu, \nu) (2\max\{\E[|X|],\E[|Y|]\} + 1) \log \left(\frac{4e(\E[|X|]+\E[|Y|]+1)}{(2\max\{\E[|X|],\E[|Y|]\} + 1) d_{LP}(\mu, \nu)} \right), $$
     where $X$ (resp. $Y$) denotes a random variable with distribution $\mu$ (resp. $\nu$).
\end{theorem}


\begin{proof}    
    Let $R > 0$. Let $X$ (resp. $Y$) be distributed according to $\mu$ (resp. $\nu$). Note that for all $t \geq 0$,
    $$ \P(|X-Y| > t) = \P(|X-Y| \geq \lfloor t \rfloor + 1) \leq \P(|X-Y| \geq 1) \leq K(X,Y), $$
    therefore,
    \begin{eqnarray*}
    \E[|X-Y|] & = & \int_0^{R} \P(|X-Y| > t) dt + \int_R^{\infty} \P(|X-Y| > t) dt \\ & \leq & R K(X,Y) + \int_R^{\infty} \P \left(|X| > \frac{t}{2} \right) dt + \int_R^{\infty} \P \left(|Y| > \frac{t}{2} \right) dt.
    \end{eqnarray*}
    Applying Lemma \ref{concentration-lem}, we obtain
    \begin{eqnarray*}
    W_1(\mu, \nu) \leq \E[|X-Y|] & \leq & R K(X,Y) + 2 \int_R^{\infty} e^{- \frac{t}{4(2\E[|X|] + 1)}} dt + 2 \int_R^{\infty} e^{- \frac{t}{4(2\E[|Y|] + 1)}} dt \\ & \leq & R K(X,Y) + 16(\E[|X|] + \E[|Y|] + 1) e^{- \frac{R}{4(2\max\{\E[|X|],\E[|Y|]\} + 1)}}.
    \end{eqnarray*}
    The above inequality being true for any random variable $X$ with distribution $\mu$ and any random variable $Y$ with distribution $\nu$, we deduce by taking infimum over all couplings that
    $$ W_1(\mu, \nu) \leq R d_{LP}(\mu, \nu) + 16(\E[|X|] + \E[|Y|] + 1) e^{- \frac{R}{4(2\max\{\E[|X|],\E[|Y|]\} + 1)}}. $$
    Choosing 
    $$ R = 4(2\max\{\E[|X|],\E[|Y|]\} + 1) \log \left(\frac{4(\E[|X|]+\E[|Y|]+1)}{(2\max\{\E[|X|],\E[|Y|]\} + 1) d_{LP}(\mu, \nu)} \right), $$
    which is nonnegative, yields the desired result.
\end{proof}

\begin{remark}

One may obtain a result for non-centered log-concave probability measures by considering the following bound,
$$ \P(|X| \geq t) \leq \P(|X-\E[X]| \geq t - |\E[X]|) \leq 2e^{- \frac{t}{2(2\E[|X - \E[X]|] + 1)}} e^{\frac{|\E[X]|}{2(2\E[|X - \E[X]|] + 1)}}, $$
valid for all log-concave random variable $X$ and $t \geq 0$, where we used Lemma \ref{concentration-lem} in the last inequality. An extra term involving $|\E[X]|$ thus appears in the non-centered case. The details are left to the reader.

\end{remark}

\begin{remark}

Theorem \ref{w1} yields the following universal bound for centered log-concave probability measures $\mu, \nu$ on $\Z$ with first absolute moment bounded by 1:
\begin{equation}\label{universal}
W_1(\mu, \nu) \leq 12 d_{LP}(\mu, \nu) \log \left( \frac{4 e}{ d_{LP}(\mu, \nu)} \right).
\end{equation}

Note that the universal bound \eqref{universal} does not hold in general without a condition on the first absolute moment, as can be seen by taking $\mu$ to be a point mass at 0 and $\nu$ a point mass at $m \in \N$ and letting $m \to +\infty$. In this case, the Wasserstein distance tends to $+\infty$ while the right-hand side of \eqref{universal} remains bounded since $d_{LP}(\mu, \nu)$ is always less than 1.

\end{remark}

\begin{remark}

In terms of statistical distances, the rate of convergence in Theorem \ref{w1} is optimal up to a log factor, as one always has $d_{LP}(\mu, \nu) \leq W_1(\mu, \nu)$. Nonetheless, it is still open whether one can remove the term $\log(1/d_{LP})$ in Theorem \ref{w1}. This remains open as well in the continuous setting.
    
\end{remark}

The next theorem demonstrates that Wasserstein distances are comparable for discrete log-concave distributions.

\begin{theorem}\label{wp}

Let $\mu$ and $\nu$ be centered log-concave probability measures on $\mathbb{Z}$, then for all $1 \leq p \leq q$,
\begin{multline*}
W_q^q(\mu, \nu) \leq 2 W_p^p(\mu, \nu) + W_p^p(\mu, \nu) \log^{q-p} \left( \frac{2^q(\E[|X|] + \E[|Y|] + 1)^q \sqrt{\Gamma(2q + 1)}}{W_p^p(\mu, \nu)} \right) \, \times \\ 8^{q-p} (2\max\{\E[|X|],\E[|Y|]\} + 1)^{q-p},
\end{multline*}
where $X$ (resp. $Y$) denotes a random variable with distribution $\mu$ (resp. $\nu$).

\end{theorem}

\begin{proof}
    Let $X$ (resp. $Y$) be distributed according to $\mu$ (resp. $\nu$). Let $R > 0$. One has
    \begin{eqnarray*}
    \E[|X-Y|^q] & = & \E[|X-Y|^{q-p+p} 1_{\{ |X-Y| < R \}}] + \E[|X-Y|^q 1_{\{ |X-Y| \geq R\}}] \\ & \leq & R^{q-p} \E[|X-Y|^{p}] + \sqrt{\P(|X-Y| \geq R) \E[|X-Y|^{2q}] },
    \end{eqnarray*}
    where we used the Cauchy-Schwarz inequality.
    Note that by Lemma \ref{moment-lem},
    $$ \E[|X-Y|^{2q}]^{\frac{1}{2q}} \leq \E[|X|^{2q}]^{\frac{1}{2q}} + \E[|Y|^{2q}]^{\frac{1}{2q}} \leq  2(\E[|X|] + \E[|Y|] + 1) \Gamma(2q + 1)^{\frac{1}{2q}}. $$
    Moreover, by Lemma \ref{concentration-lem},
    \begin{eqnarray*} \P(|X-Y| \geq R) \leq \P \left( |X| \geq \frac{R}{2} \right) + \P \left( |Y| \geq \frac{R}{2} \right) & \leq & 2 e^{- \frac{R}{4(2\E[|X|] + 1)}} + 2 e^{- \frac{R}{4(2\E[|Y|] + 1)}} \\ & \leq & 4 e^{- \frac{R}{4(2\max\{\E[|X|],\E[|Y|]\} + 1)}}.
    \end{eqnarray*}
    Combining the above and taking infimum over all couplings yield
    $$ W_q^q(\mu, \nu) \leq R^{q-p} W_p^p(\mu, \nu) + 2^{q}(\E[|X|] + \E[|Y|] + 1)^{q} \sqrt{\Gamma(2q + 1)} 2 e^{- \frac{R}{8(2\max\{\E[|X|],\E[|Y|]\} + 1)}}. $$
    The result follows by choosing 
    $$ R = 8 (2\max\{\E[|X|],\E[|Y|]\} + 1) \log \left( \frac{2^q(\E[|X|] + \E[|Y|] + 1)^q \sqrt{\Gamma(2q + 1)}}{W_p^p(\mu, \nu)} \right), $$
    which is nonnegative since by Lemma \ref{moment-lem} and log-convexity of the Gamma function,
    \begin{eqnarray*}
    W_p^p(\mu,\nu) \leq \left( \E[|X|^p]^{\frac{1}{p}} + \E[|Y|^p]^{\frac{1}{p}} \right)^p & \leq & 2^p(\E[|X|] + \E[|Y|] + 1)^p \Gamma(p+1) \\ & \leq & 2^q(\E[|X|] + \E[|Y|] + 1)^q \sqrt{\Gamma(2q+1)}.
    \end{eqnarray*}
\end{proof}

\begin{remark}

Theorem \ref{wp} yields the following universal bound for centered log-concave probability measures $\mu, \nu$ on $\Z$ with first absolute moment bounded by 1:
\begin{equation*}
W_q^q(\mu, \nu) \leq 24^{q-p} W_p^p(\mu, \nu) \log^{q-p} \left( \frac{6^{q} \sqrt{\Gamma(2q + 1)}}{W_p^p(\mu, \nu)} \right) + 2 W_p^p(\mu, \nu).
\end{equation*}

\end{remark}

\begin{remark}

In terms of statistical distances, Theorem \ref{wp} yields a rate of convergence of the form
$$ W_p(\mu, \nu) \leq W_q(\mu, \nu) \leq C \, W_p^{\frac{p}{q}}(\mu, \nu) \log^{1-\frac{p}{q}}(1/W_p(\mu, \nu)), $$
for some constant $C$ depending on $p$, $q$, and the first absolute moment of $\mu$ and $\nu$. It is still open whether one can remove the term $\log^{1-\frac{p}{q}}(1/W_p)$ and improve the power $p/q$ in the term $W_p^{p/q}$ in Theorem \ref{wp}. This remains open as well in the continuous setting.

\end{remark}

Let us now turn to $f$-divergences. Considering $f$-divergences, such as the Kullback-Leibler divergence, the main question lies in figuring out the distribution of the reference measure. In general, if the support of a measure $\mu$ is not included in the support of a measure $\nu$, then $D(\mu || \nu) = + \infty$. It turns out that our choice of reference measure needs not be log-concave. Given $a > 0$ and $c \geq 1$, let us introduce the following class of functions:
$$ \mathcal{Q}(a,c) = \{q \colon \Z \to [0,1] : \forall k \in \supp(q), \, \log \left( \frac{1}{q(k)} \right) \leq ak^2 + \log(c) \}, $$
where $\supp(q) = \{k \in \Z : q(k) > 0\}$.

Before stating our next result, let us note that important distributions belong to such a class.

\begin{remark}

The symmetric Poisson distribution with variance 1, whose probability mass function is
$$ q(k) = C \frac{\lambda^{|k|}}{|k|!}, \quad k \in \Z, $$
with $\lambda > 0$ such that $\sum_{k \in \Z} k^2 q(k) = 1$ and $C = (2 e^{\lambda} - 1)^{-1}$ being the normalizing constant, belongs to $\mathcal{Q}(1+\log(4), 2e -1)$. Indeed, since
$$ 1 = \sum_{k \in \Z} k^2 q(k) = \frac{2e^{\lambda}}{2 e^{\lambda} - 1}\lambda(1 + \lambda), $$
then one may choose $\lambda \in [1/4,1]$. Therefore, using $|k|! \leq |k|^{|k|}$,
$$ 0 \leq \log \left( \frac{1}{q(k)} \right) = \log(|k|!) + |k| \log \left( \frac{1}{\lambda} \right) + \log(2 e^{\lambda} - 1) \leq (1+\log(4)) k^2 + \log(2e - 1). $$

One may also note that the symmetric geometric distribution with variance 1 and the discretized Gaussian distribution with variance 1 (whose p.m.f. is of the form $q(k) = Ce^{-\lambda k^2}$) belong to $\mathcal{Q}(a,c)$ for some numerical constants $a,c>0$. The above three measures are natural candidates as a reference measure for Kullback-Leibler divergence. 

As for examples of non-log-concave distributions, consider p.m.f. of the form $C e^{-\lambda k^{\alpha}}$, for $\alpha \in (0,1)$.

As for an example of a log-concave distribution that does not belong to $\mathcal{Q}(a,c)$, consider the discrete uniform distribution on $\{-m ,\dots, m\}$ with $m > (c-1)/2$.

\end{remark}

The next result provides a comparison between total variation distance and $f$-divergences. The result is general as it holds for arbitrary convex function $f$, however the statement is not in a closed form formula. We state it as a lemma, and then apply it to two specific convex functions, yielding a comparison with Kullback-Leibler divergence and $\chi^2$-divergence.

\begin{lemma}\label{f-div}

Let $f \colon [0, +\infty) \to \R$ be a convex function such that $f(1) = 0$. Let $a > 0$ and $c \geq 1$. Let $\nu$ be a measure on $\Z$ whose p.m.f. $q$ belongs to the class $\mathcal{Q}(a,c)$. Let $\mu$ be a measure on $\Z$ with p.m.f. $p$ whose support is included in the support of $\nu$. Then, denoting by $Y$ a random variable with distribution $\mu$ and $W=p(Y)/q(Y)$,
\begin{multline*}
d_f(\mu||\nu) \leq \inf_{R \geq c} \left[ \left(\max\{f(0),0\} + \frac{f(R)}{R-1} \right) d_{TV}(\mu, \nu) \, + \right. \\ \left. \sqrt{\E \left[ \left(\frac{f(W)}{W} \right)^2 1_{\{W > 1\}} \right] \P \left( |Y| > \sqrt{\frac{1}{a} \log \left( \frac{R}{c} \right)} \right)} \right].
\end{multline*}

\end{lemma}

\begin{proof}
The idea of proof comes from \cite{MM} (see also \cite{MP}). Denote by $p$ (resp. $q$) the p.m.f. of $\mu$ (resp. $\nu$). Denote by $Y$ a random variable with p.m.f. $p$, by $Z$ a random variable with p.m.f. $q$, and denote
$$ X = \frac{p(Z)}{q(Z)}, \qquad W = \frac{p(Y)}{q(Y)}. $$
Using identity \eqref{d-tv-dis}, one has
\begin{equation}\label{tot-var}
\E[|X-1|] = d_{TV}(\mu, \nu).
\end{equation}
Let $R \geq 1$ and write
$$ d_f(\mu||\nu) = \E[f(X)] = \E[f(X) 1_{\{X < 1\}}] + \E[f(X) 1_{\{1 \leq X \leq R\}}] + \E[f(X) 1_{\{X > R\}}]. $$
Let us bound all three parts. For the first part, since $f$ is convex and $f(1) = 0$, it holds that for all $x \in [0,1]$, 
$$ f(x) \leq f(0)|x-1| \leq \max\{f(0),0\} |x-1|. $$ Therefore, using \eqref{tot-var},
\begin{equation}\label{ineq-1}
\E[f(X) 1_{\{X < 1\}}] \leq \max\{f(0),0\} \E[|X-1| 1_{\{X < 1\}}] \leq \max\{f(0),0\} d_{TV}(\mu, \nu).
\end{equation}
For the second part, since $f$ is convex and $f(1) = 0$, it holds that for all $x \in [1,R]$,
$$ f(x) \leq \frac{f(R)}{R-1} (x-1). $$
Hence, using \eqref{tot-var},
\begin{equation}\label{ineq-2}
\E[f(X) 1_{\{1 \leq X \leq R\}}] \leq \frac{f(R)}{R-1} \E[(X-1) 1_{\{1 \leq X \leq R\}}] \leq \frac{f(R)}{R-1} d_{TV}(\mu, \nu).
\end{equation}
For the last part, note that
\begin{equation}\label{ineq-33}
\E[f(X) 1_{\{X > R\}}] = \E \left[ \frac{f(W)}{W} 1_{\{W > R\}} \right] \leq \sqrt{\E \left[ \left(\frac{f(W)}{W} \right)^2 1_{\{W > 1\}} \right]} \sqrt{\P(W > R)},
\end{equation}
where we used the Cauchy-Schwarz inequality. It remains to upper bound $\P(W > R)$. Using that $q \in \mathcal{Q}(a,c)$ and $\|p\|_{\infty} \leq 1$, we have for all $R \geq c$,
\begin{eqnarray}\label{ineq-3}
\P(W > R) = \P \left( \frac{p(Y)}{q(Y)} > R \right) \leq \P \left( \log(c) + a Y^2 > \log(R) \right) = \P \left( |Y| > \sqrt{\frac{1}{a} \log \left( \frac{R}{c} \right)} \right).
\end{eqnarray}
The result follows by combining \eqref{ineq-1}, \eqref{ineq-2}, \eqref{ineq-33}, and \eqref{ineq-3}, and by taking infimum over all $R \geq c$.
\end{proof}

Applying Lemma \ref{f-div} to the convex function $f(x)=x\log(x)$, $x \geq 0$, yields a comparison between total variation distance and Kullback-Leibler divergence.

\begin{theorem}\label{divergence}
Let $a > 0$ and $c \geq 2$. Let $\nu$ be a measure on $\Z$ whose p.m.f. belongs to the class $\mathcal{Q}(a,c)$. Let $\mu$ be a log-concave measure on $\Z$ whose support is included in the support of $\nu$. Then,
\begin{equation}\label{div-eq}
D(\mu || \nu) \leq d_{TV}(\mu, \nu) \left(32 a (2\sqrt{\Var(Y)} + 1)^2 \log^2 \left( \frac{ \sqrt{2}( A + B)}{d_{TV}(\mu, \nu)} \right) + 2 \log(c) + 1 \right),
\end{equation}
where $Y$ denotes a random variable with distribution $\mu$, and
$$ A = \log(c) + 10a + 41a\Var(Y), $$
$$ B = \frac{5}{2} + (1 + 12\Var(Y))^{\frac{1}{4}} \log\left( 2\pi e \left( \Var(Y) + \frac{1}{12} \right) \right). $$
\end{theorem}

\begin{proof}
Recall the notation $W=p(Y)/q(Y)$ from Lemma \ref{f-div}. With the choice of convex function $f(x) = x\log(x)$, $x \geq 0$, Lemma \ref{f-div} tells us that for all $R \geq c \geq 2$,
$$ D(\mu||\nu) \leq  2\log(R) d_{TV}(\mu, \nu) + \sqrt{\E[|\log(W)|^{2}]} \sqrt{\P \left( |Y| > \sqrt{\frac{1}{a} \log \left( \frac{R}{c} \right)} \right)}. $$
Using Lemma \ref{concentration-lem} and the bound $\E[|Y|] \leq \sqrt{\Var(Y)}$, we deduce
$$ D(\mu||\nu) \leq  2\log(R) d_{TV}(\mu, \nu) + \sqrt{\E[|\log(W)|^{2}]} \sqrt{2}e^{-\frac{1}{4} \frac{1}{2\sqrt{\Var(Y)} + 1} \sqrt{\frac{1}{a} \log \left( \frac{R}{c} \right)} }. $$
Next, let us upper bound the term $\E[|\log(W)|^{2}]^{1/2}$. On one hand, since $q \in \mathcal{Q}(a,c)$,
\begin{eqnarray}\label{first}
\E[|\log(q(Y))|^{2}] & \leq & \E[\left( \log(c) + a Y^2 \right)^2] \nonumber \\ & = & \log^2(c) + 2a \log(c) \Var(Y) + a^2 \E[Y^4] \nonumber \\ & \leq & \log^2(c) + 2a \log(c) \Var(Y) + 24 a^2 (2\sqrt{\Var(Y)} + 1)^4 \nonumber \\ & := & \widetilde{A},
\end{eqnarray}
where we used Lemma \ref{moment-lem} in the last inequality. On the other hand, by Lemma \ref{second-bound},
\begin{eqnarray}\label{second}
\E[|\log(p(Y))|^2] & \leq & 4 \left( 4 e^{-2} + 1 + \frac{H^2(X)}{\|p\|_{\infty}} \right) \nonumber \\ & \leq & 4 \left( 4 e^{-2} + 1 + \frac{1}{4} \log^2\left( 2\pi e \left( \Var(Y) + \frac{1}{12} \right) \right) \sqrt{1 + 12\Var(Y)} \right) \nonumber \\ & := & \widetilde{B},
\end{eqnarray}
where we used Lemmas \ref{maximum-bound} and \ref{bound-ent} in the last inequality. Therefore, combining \eqref{first} and \eqref{second},
\begin{eqnarray}\label{1st}
\E[|\log(W)|^{2}]^{\frac{1}{2}} \leq \E[|\log(p(Y))|^2]^{\frac{1}{2}} + \E[|\log(q(Y))|^{2}]^{\frac{1}{2}} \leq \sqrt{\widetilde{A}} + \sqrt{\widetilde{B}},
\end{eqnarray}
implying
$$ D(\mu || \nu) \leq 2\log(R) d_{TV}(\mu, \nu) + (\sqrt{\widetilde{A}} + \sqrt{\widetilde{B}}) \sqrt{2}e^{-\frac{1}{4} \frac{1}{2\sqrt{\Var(Y)} + 1} \sqrt{\frac{1}{a} \log \left( \frac{R}{c} \right)} }. $$
Putting $t =  \sqrt{ \log(\frac{R}{c})}$, the above inequality reads
$$ D(\mu || \nu) \leq 2\log(c) d_{TV}(\mu, \nu) + 2 d_{TV}(\mu, \nu) t^2 + (\sqrt{\widetilde{A}} + \sqrt{\widetilde{B}}) \sqrt{2}e^{-\frac{1}{4} \frac{1}{2\sqrt{\Var(Y)} + 1} \frac{t}{\sqrt{a}} }. $$
Choosing $t = 4(2\sqrt{\Var(Y)} + 1) \sqrt{a} \log \left( \frac{ \left( \sqrt{\widetilde{A}} + \sqrt{\widetilde{B}} \right) \sqrt{2} }{d_{TV}(\mu, \nu)} \right)$, which is nonnegative, and using the bounds
$$ \sqrt{\widetilde{A}} \leq \sqrt{(\log(c) + a\Var(Y))^2 + 96 a^2 (4\Var(Y) + 1)^2} \leq \log(c) + a\Var(Y) + 10a(4\Var(Y) + 1), $$
$$ \sqrt{\widetilde{B}} \leq \sqrt{\frac{31}{5}} + (1 + 12\Var(Y))^{\frac{1}{4}} \log\left( 2\pi e \left( \Var(Y) + \frac{1}{12} \right) \right)  $$
yields the desired result.
\end{proof}

\begin{remark}

Theorem \ref{divergence} yields the following bound depending only on $a$ and $c$ for centered log-concave probability measures $\mu$ on $\Z$ whose variance is bounded by 1:
\begin{equation*}
D(\mu || \nu) \leq d_{TV}(\mu, \nu) \left( 288a \log^2 \left( \frac{ \sqrt{2}( 9 + \log(c) + 51a )}{d_{TV}(\mu, \nu)} \right) + 2 \log(c) + 1 \right).
\end{equation*}

\end{remark}

\begin{remark}

A comparison between total variation distance and Kullback-Leibler divergence similar to Theorem \ref{divergence} (with given fixed constants $a > 0$ and $c \geq 2$) does not hold if the reference measure $\nu$ is an arbitrary log-concave measure. For example, consider $\mu$ to be the Dirac measure at 0, and $\nu$ to be the discrete uniform measure on $\{-m, \dots, m\}$, for $m \in \N$. Then, $D(\mu || \nu) = \log(2m + 1) \to +\infty$ as $m \to +\infty$, while the right-hand side of \eqref{div-eq} remains bounded. Nonetheless, we believe the result to be true for arbitrary log-concave measure $\nu$ under a constraint of bounded variance.

\end{remark}

\begin{remark}

In terms of statistical distances, Theorem \ref{divergence} yields a rate of convergence of the form
$$ d_{TV}^2(\mu, \nu) \leq 2 D(\mu || \nu) \leq C \, d_{TV}(\mu, \nu) \log^2(1/d_{TV}(\mu, \nu)), $$
for some constant $C$ depending on $a$, $c$, and the variance of $\mu$. It is still open whether one can remove the term $\log^2(1/d_{TV})$ in Theorem \ref{divergence} and improve the power in the term $d_{TV}$. This remains open as well in the continuous setting, where only the Gaussian measure as the reference measure has been studied.

\end{remark}

As a last illustration, the next result provides a comparison between total variation distance and $\chi^2$-divergence, under an extra moment assumption. For simplicity, we restrict the statement to log-concave measures with bounded variance.

\begin{theorem}
Let $a > 0$ and $c \geq 1$. Let $\nu$ be a measure on $\Z$ whose p.m.f. belongs to the class $\mathcal{Q}(a,c)$. Let $\mu$ be a centered log-concave measure on $\Z$ with variance bounded by 1 and whose support is included in the support of $\nu$. Under the moment assumption
$$ \E[e^{2aY^2}] < +\infty, $$
where $Y$ denotes a random variable with distribution $\mu$, one has
$$ d_{\chi^2}(\mu || \nu) \leq c \left( d_{TV}(\mu, \nu) + \sqrt{d_{TV}(\mu, \nu)} \right) + c \sqrt{\E[e^{2a Y^2}]} \sqrt{2 }e^{-\frac{1}{12} \sqrt{\frac{1}{a} \log \left( 1 + \frac{1}{\sqrt{d_{TV}(\mu, \nu)}} \right)} }. $$
\end{theorem}

\begin{proof}
Recall the notation $W=p(Y)/q(Y)$ from Lemma \ref{f-div}. With the choice of convex function $f(x) = (x-1)^2$, $x \geq 0$, Lemma \ref{f-div} tells us that for all $R \geq c$,
$$ d_{\chi^2}(\mu || \nu) \leq R d_{TV}(\mu, \nu) + \sqrt{\E \left[ \frac{(W-1)^4}{W^2} 1_{\{W > 1\}} \right]} \sqrt{2} e^{-\frac{1}{12} \sqrt{\frac{1}{a} \log \left( \frac{R}{c} \right)} }. $$
Note that
$$ \E \left[ \frac{(W-1)^4}{W^2} 1_{\{W > 1\}} \right] \leq \E \left[ W^2 \right] \leq \E \left[ \frac{1}{q^2(Y)} \right] \leq c^2 \E[e^{2a Y^2}], $$
where the last inequality comes from $q \in \mathcal{Q}(a,c)$. Therefore,
$$ d_{\chi^2}(\mu || \nu) \leq R d_{TV}(\mu, \nu) + c \sqrt{\E[e^{2a Y^2}]} \sqrt{2 }e^{-\frac{1}{12} \sqrt{\frac{1}{a} \log \left( \frac{R}{c} \right)} }. $$
Choosing $R = c \left(1 + \frac{1}{\sqrt{d_{TV}(\mu, \nu)}} \right)$ yields the desired result.
\end{proof}

\section*{Acknowledgments.}
The authors are grateful to the referee for their valuable comments.

\vskip3cm

\noindent Arnaud Marsiglietti \\
Department of Mathematics \\
University of Florida \\
Gainesville, FL 32611, USA \\
a.marsiglietti@ufl.edu

\vspace{0.8cm}

\noindent Puja Pandey \\
Department of Statistics and Applied Probability \\
University of California Santa Barbara \\
Goleta, CA 93106, USA \\
pujapandey@ucsb.edu

\end{document}